\documentclass[12pt,thmsa]{article}
\usepackage{amssymb}
\usepackage{amsfonts}
\usepackage{enumerate}
\usepackage{amsmath}
\usepackage{cite}
\usepackage[top=3.8cm,bottom=3.8cm,textwidth=16cm,centering,a4paper]{geometry}
\usepackage{tikz}
\usepackage{caption2}
\usepackage{subfigure}
\usepackage{tcolorbox}
\usepackage{hyperref}
\usepackage{cleveref}
\usepackage{indentfirst}

\allowdisplaybreaks

\newtheorem{theorem}{Theorem}[section]

\newtheorem{proposition}[theorem]{Proposition}
\newtheorem{corollary}[theorem]{Corollary}

\newtheorem{lemma}[theorem]{Lemma}

\newenvironment{proof}[1][Proof]{\noindent\textbf{#1.} }{\ \rule{0.5em}{0.5em}}

\begin{document}

\title{Uniqueness of Ground State Solutions for a Defocusing Hartree
Equation via Inverse Optimal Problems}
\date{}
\author{Yavdat~Il'yasov$^{a}$\thanks{%
E-mail address: ilyasov02@gmail.com (Yavdat~Il'yasov)}, Juntao Sun$^{b}$%
\thanks{%
E-mail address: jtsun@sdut.edu.cn (Juntao Sun)}, Nur~Valeev$^{a}$\thanks{%
E-mail address: valeevnf@mail.ru (Nur~Valeev)}, Shuai Yao$^{b}$\thanks{%
E-mail address: shyao@sdut.edu.cn (Shuai Yao)} \\
$^{a}${\footnotesize \emph{Institute of Mathematic, Ufa Federal Research
Centre, RAS, Ufa 450008, Russia}}\\
$^{b}${\footnotesize \emph{School of Mathematics and Statistics, Shandong
University of Technology, Zibo 255049, PR China}}}
\maketitle

\begin{abstract}
We study a generalized defocusing Hartree equation with nonlocal exchange
potential and repulsive Hartree--Fock interaction. Using an inverse optimal
problem (IOP) approach, we prove the existence and uniqueness of ground
state solutions. Additionally, we establish the existence of principal
solutions, their continuous dependence on parameters, and a dual variational
formulation. The IOP method provides a systematic framework for addressing
inverse problems in nonlocal Schr\"{o}dinger operators and offers new
insights into the structure of solutions for defocusing Hartree-type
equations.
\end{abstract}



\section{Introduction and statement of main results}

In this paper, we study the following generalized defocusing Hartree
equation
\begin{equation}
-\Delta u+V(x)u-S[\rho ]u+\gamma (|x|^{-\mu }\ast |u|^{2})u=\lambda u,\quad
\forall x\in \mathbb{R}^{N},  \label{eq:M}
\end{equation}%
where the exchange potential is given by
\begin{equation*}
S[\rho ]u(x)=\int_{\mathbb{R}^{N}}\frac{\rho (x,y)\,u(y)}{|x-y|^{\mu }}dy,
\end{equation*}%
the Hartree--Fock term is defined by
\begin{equation*}
(|x|^{-\mu }\ast |u|^{2})(x)=\int_{\mathbb{R}^{N}}\frac{|u(y)|^{2}}{%
|x-y|^{\mu }}dy,
\end{equation*}%
and the external potential $V\in C(\mathbb{R}^{N})$ satisfies the conditions
\begin{equation}
\lim_{|x|\rightarrow +\infty }V(x)=+\infty ,\qquad \text{and}\qquad V(x)\geq
a_{V}>0,\quad \forall x\in \mathbb{R}^{N},  \label{cond:V}
\end{equation}%
for some constant $a_{V}>0$. Here $(-\Delta )$ is the standard Laplacian, $%
N\geq 1$, $\lambda \in \mathbb{R}$, $\gamma >0$, and $0\leq \mu <\min
\{N,4\} $.

We assume that the interaction kernel $\rho$ is symmetric, $\rho(x,y) =
\rho(y,x)$, and belongs to the space $\mathbb{L}_\mu^2(\mathbb{R}^N \times
\mathbb{R}^N)$ with finite norm
\begin{equation*}
\| \rho \|_{\mathbb{L}_{\mu}^{2}}^{2} := \int_{\mathbb{R}^{N}} \int_{\mathbb{%
R}^{N}} \frac{|\rho(x,y)|^{2}}{|x-y|^{\mu}}dxdy < +\infty.
\end{equation*}

Equation (\ref{eq:M}) is considered in the distributional sense, and we seek
weak solutions in the space $W:=W^{1,2}(\mathbb{R}^{N})\cap L_{V}^{2}(%
\mathbb{R}^{N})$ with the norm
\begin{equation*}
\Vert u\Vert _{W}=\left[ \int_{\mathbb{R}^{N}}(|\nabla u|^{2}+V(x)|u|^{2})dx%
\right] ^{1/2},
\end{equation*}%
where $W^{1,2}(\mathbb{R}^{N})$ denotes the standard Sobolev space, and
\begin{equation*}
L_{V}^{2}:=\left\{ u\in L^{2}(\mathbb{R}^{N})\mid \Vert u\Vert
_{L_{V}^{2}}^{2}:=\int_{\mathbb{R}^{N}}V(x)u^{2}dx<+\infty \right\} .
\end{equation*}%
Equation (\ref{eq:M}) has a variational form with the energy functional $%
E_{\lambda }\in C^{1}(W)$ given by
\begin{eqnarray*}
E_{\lambda }(u) &=&\frac{1}{2}\left( \int_{\mathbb{R}^{N}}|\nabla
u|^{2}dx+\int_{\mathbb{R}^{N}}V|u|^{2}dx-\int_{\mathbb{R}^{N}}S[\rho
]u^{2}dx-\lambda \int_{\mathbb{R}^{N}}|u|^{2}dx\right) \\
&&+\frac{1}{4}\int_{\mathbb{R}^{N}}\left( |x|^{-\mu }\ast |u|^{2}\right)
|u|^{2}dx.
\end{eqnarray*}%
We define a weak solution $u_{\lambda }\in W$ of equation (\ref{eq:M}) as a
\textit{ground state} if $E_{\lambda }(u_{\lambda })\leq E_{\lambda }(w)$
holds for every weak solution $w\in W$ of equation (\ref{eq:M}). A solution $%
u_{\lambda }\in W$ of equation (\ref{eq:M}) is said to be \textit{stable} if
$D_{uu}E_{\lambda }(\hat{u}_{g})(h,h)\geq 0$ for any $h\in W$ (cf. \cite%
{Dupaigne}). Hereafter, $D_{u}E_{\lambda }(u)(\cdot )$, $u\in W$ denotes the
Fr\'{e}chet differential of $E_{\lambda }(u)$.

Under condition (\ref{cond:V}) and assuming that $\rho \in \mathbb{L}_{\mu
}^{2}$, the general theory of linear operators (see, e.g., \cite{Kato,
SimonReed}) implies that the operator
\begin{equation*}
\mathcal{L}[\rho ]:=-\Delta +V-S[\rho ],
\end{equation*}%
with domain $D(\mathcal{L}[\rho ]):=W^{2,2}(\mathbb{R}^{N})$, defines a
self-adjoint operator. Consequently, its spectrum consists of an infinite
sequence of eigenvalues $\{\lambda _{i}(\rho )\}_{i=1}^{\infty }$, ordered
as
\begin{equation*}
\lambda _{1}(\rho )\leq \lambda _{2}(\rho )\leq \dots \leq \lambda _{k}(\rho
)\leq \dots ,\quad \lim_{k\rightarrow \infty }\lambda _{k}(\rho )=+\infty .
\end{equation*}%
Furthermore, the principal eigenvalue satisfies
\begin{equation}
-\infty <\lambda _{1}(\rho )=\inf_{\psi \in C_{0}^{\infty }(\mathbb{R}%
^{N})\setminus \{0\}}\frac{(\mathcal{L}[\rho ]\psi ,\psi )}{(\psi ,\psi )}.
\label{eq:lambda1}
\end{equation}

We focus on the so-called principal solutions of the equation. For given $%
\lambda \in \mathbb{R}$ and $\rho \in \mathbb{L}_{\mu }^{2}$, a weak
solution $\hat{u}:=\hat{u}(\lambda ,\rho )\in W$ of equation (\ref{eq:M}) is
called \textit{principal} if
\begin{equation*}
\phi _{1}(x):=\frac{\hat{u}(x)}{\Vert \hat{u}\Vert _{L^{2}}},\text{ }\forall
x\in \mathbb{R}^{N},
\end{equation*}%
is the principal eigenfunction of the nonlocal Schr\"{o}dinger operator $%
\mathcal{L}[\hat{\rho}]$ associated with eigenvalue $\lambda $, with
\begin{equation*}
\hat{\rho}(x,y):=\rho (x,y)-\hat{u}(x)\hat{u}(y),\quad \text{a.e. in }%
\mathbb{R}^{N}\times \mathbb{R}^{N}.
\end{equation*}

Equation (\ref{eq:M}), for various values of its parameters (including $\rho
\equiv 0$ and $\mu =0$), models systems with long-range interactions and
belongs to a wide class of nonlinear Schr\"{o}dinger equations with nonlocal
interactions \cite{Benguria, Bethe, Hartree, Kohn, Lions, LiebSimon,
Penrose, Slater}. In the case $\gamma <0$, this equation is known as the
focusing equation, also called the Choquard and Hartree--Fock equation or
the Schr\"{o}dinger--Newton equation, and it has been extensively studied
over the past several decades (see the surveys \cite{Arora, LeLions,
MorozSurv} and the references therein).

The defocusing equation (\ref{eq:M}), corresponding to the case $\gamma >0$,
constitutes a variant of the Hartree equation \cite{Hartree,Hartree57,
Lions87} that incorporates a nonlocal repulsive (defocusing) interaction
term. Unlike the focusing case, these equations have been extensively
studied in mathematical physics as models for dispersive wave dynamics,
where solutions typically spread out rather than concentrate or exhibit
blow-up (see, e.g., \cite{Benguria, Bourgain, Cao, Ginibre, Miao2007, Tao,
Wang}). They consist of intricate nonlinear partial differential equations
involving nonlocal potentials, combining the mean-field framework of
Hartree--Fock theory with repulsive interactions, and give rise to a rich
variety of dynamical behaviors that remain an active area of research in
nonlinear analysis. The present work will further demonstrate that the
investigation of such Hartree-type equations is relevant to the study of
inverse spectral problems \cite{chadan, glad}.

This paper explores Hartree-type equations by leveraging the inverse optimal
problem approach proposed by Ilyasov and Valeev in \cite{ValIl_JDE,
ValIl_PhysD, ValIl2025, ValeevIl2018}. This method is characterized by the
following key attributes: (1) it offers a specialized technique for
constructing solutions; (2) it enables solutions to be expressed via
variational formulas linked to inverse optimal problems; and (3) it develops
a dedicated methodology for proving the uniqueness of solutions to nonlinear
differential equations. Research on inverse spectral problems has primarily
centered on Sturm-Liouville operators (see, e.g., \cite{ValIl2025, Qi, Wei,
ZhangM}). Notably, the inverse nodal problem for Sturm-Liouville operators
has also received attention in recent literature \cite{CMX,HWXZ}. This class
of inverse problems for nonlocal operators has not been previously studied.

Establishing uniqueness for equations with nonlocal interactions presents
significant challenges. Classical techniques, such as symmetric decreasing
rearrangement and the moving plane method, are inapplicable due to the
presence of the exchange potential $S[\rho ]$ and the nonlocal defocusing
term $\left( |x|^{-\mu }\ast |u|^{2}\right) u$ (cf. \cite{Brascamp, Kwong,
Lieb}). Additionally, the associated energy functional $E_{\lambda }(u)$
lacks convexity (see, e.g., \cite{ChoquardSt}). Moreover, it fails to
exhibit phase invariance, as $E_{\lambda }(|u|)\neq E_{\lambda }(u)$ is
possible for some $u\in W$.

A central component of this approach is the following \textit{inverse
optimal problem} (IOP):

\medskip \noindent \textit{($\mathcal{P}$): given $\lambda \in \mathbb{R}$
and $\bar{\rho}\in \mathbb{L}_{\mu }^{2}$, find a minimizer $\hat{\rho}:=%
\hat{\rho}(\lambda ,\bar{\rho})\in \mathbb{L}_{\mu }^{2}$ such that $\lambda
=\lambda _{1}(\rho )$ and}
\begin{equation}
\hat{\mathcal{P}}(\lambda ,\bar{\rho}):=\Vert \bar{\rho}-\hat{\rho}\Vert _{%
\mathbb{L}_{\mu }^{2}}^{2}=\min \left\{ \Vert \bar{\rho}-\rho \Vert _{%
\mathbb{L}_{\mu }^{2}}^{2}\mid \rho \in \mathbb{L}_{\mu }^{2},\ \lambda
=\lambda _{1}(\rho )\right\} .  \label{eq:p}
\end{equation}%
Obviously, $\hat{\rho}=\bar{\rho}$ and $\hat{\mathcal{P}}(\lambda ,\bar{\rho}%
)=0$ if $\lambda =\lambda _{1}(\bar{\rho})$, and if $\lambda \leq \lambda
_{1}(\bar{\rho})$, equation (\ref{eq:M}) has no nontrivial weak solution.

Our first theorem establishes the existence of a principal solution and
relates it to the minimizer of the IOP.

\begin{theorem}
\label{thm1} Assume that $0 < \mu < \min\{N, 4\}$, $\bar{\rho} \in \mathbb{L}%
_\mu^2$, and $\lambda > \lambda_1(\bar{\rho})$. Then:

\begin{enumerate}
\item[$(i)$] Equation (\ref{eq:M}) admits a \emph{principal solution} $\hat{u%
}\in W$.

\item[$(ii)$] There exists a \emph{unique} minimizer $\hat{\rho}$ of the
inverse optimal problem (\ref{eq:p}), and it can be recovered pointwise
almost everywhere via
\begin{equation*}
\hat{\rho}(x,y)=\bar{\rho}(x,y)-\hat{u}(x)\hat{u}(y)\quad \text{a.e. in }%
\mathbb{R}^{N}\times \mathbb{R}^{N}.
\end{equation*}

\item[$(iii)$] The principal solution $\hat{u}\in W$ is uniquely determined
almost everywhere (up to its sign) by
\begin{equation*}
|\hat{u}(x)|=\sqrt{\bar{\rho}(x,x)-\hat{\rho}(x,x)}\quad \text{a.e. in }%
\mathbb{R}^{N}.
\end{equation*}
\end{enumerate}
\end{theorem}

Hereafter, a function $w \in L^1_{\text{loc}}(\mathbb{R}^N)$ is said to be
\emph{positive} (respectively, \emph{negative}) in $\mathbb{R}^N$ if $w(x) >
0$ (respectively, $w(x) < 0$) for almost every $x \in \mathbb{R}^N$.

Our second theorem concerns the existence and uniqueness of a ground state
when the interaction kernel is nonnegative.

\begin{theorem}
\label{thm2} Assume that $0<\mu <\min \{N,4\}$, $\bar{\rho}\in \mathbb{L}%
_{\mu }^{2}$, and $\lambda >\lambda _{1}(\bar{\rho})$. Then equation (\ref%
{eq:M}) admits a stable \emph{ground state}. Moreover, if $\bar{\rho}\geq 0$
a.e. in $\mathbb{R}^{N}\times \mathbb{R}^{N}$, the following hold:

\begin{enumerate}
\item[$(i)$] The ground state is unique and positive. It can be recovered
pointwise almost everywhere via
\begin{equation*}
\hat{u}(x) = \sqrt{\bar{\rho}(x,x) - \hat{\rho}(x,x)} \quad \text{a.e. in }
\mathbb{R}^N.
\end{equation*}

\item[$(ii)$] The principal solution $\hat{u}\in W$ is positive and
coincides with the ground state of equation (\ref{eq:M}).
\end{enumerate}
\end{theorem}

In the following theorem, we show that the set of principal solutions forms,
in a certain sense, a continuous branch of solutions.

\begin{theorem}
\label{thm3} Assume that $0<\mu <\min \{N,4\}$.

\begin{enumerate}
\item[$(i)$] Let $\lambda \geq \lambda _{1}(\bar{\rho})$ and suppose $\rho
_{n}\rightarrow \bar{\rho}$ in $\mathbb{L}_{\mu }^{2}$ as $n\rightarrow
\infty $. Then there exists a subsequence, still denoted by $(\rho _{n})$,
and a solution $\hat{u}(\lambda ,\bar{\rho})$ such that
\begin{equation*}
\hat{u}(\lambda ,\rho _{n})\rightarrow \hat{u}(\lambda ,\bar{\rho})\quad
\text{in }W\quad \text{as }n\rightarrow +\infty .
\end{equation*}%
Moreover:

\begin{itemize}
\item If $\lambda >\lambda _{1}(\bar{\rho})$, then $\hat{u}(\lambda ,\bar{%
\rho})$ is a principal solution.

\item If $\lambda =\lambda _{1}(\bar{\rho})$, then $\hat{u}(\lambda ,\bar{%
\rho})=0$.
\end{itemize}

\item[$(ii)$] Let $\bar{\rho}\in \mathbb{L}_{\mu }^{2}$ and suppose $\lambda
_{n}\rightarrow \lambda \in \lbrack \lambda _{1}(\bar{\rho}),+\infty )$ as $%
n\rightarrow \infty $. Then there exists a subsequence, still denoted by $%
(\lambda _{n})$, and a solution $\hat{u}(\lambda ,\bar{\rho})$ such that
\begin{equation*}
\hat{u}(\lambda _{n},\bar{\rho})\rightarrow \hat{u}(\lambda ,\bar{\rho}%
)\quad \text{in }W\quad \text{as }n\rightarrow +\infty .
\end{equation*}%
Moreover:

\begin{itemize}
\item If $\lambda >\lambda _{1}(\bar{\rho})$, then $\hat{u}(\lambda ,\bar{%
\rho})$ is a principal solution.

\item If $\lambda =\lambda _{1}(\bar{\rho})$, then $\hat{u}(\lambda ,\bar{%
\rho})=0$.
\end{itemize}

\item[$(iii)$] The following continuity properties hold:

\begin{itemize}
\item For any $\lambda \geq \lambda _{1}(\bar{\rho})$, the map $\mathbb{L}%
_{\mu }^{2}\ni \rho \mapsto \hat{u}(\lambda ,\rho )\in W$ is continuous at $%
\bar{\rho}\in \mathbb{L}_{\mu }^{2}$.

\item For any $\bar{\rho}\in \mathbb{L}_{\mu }^{2}$, the map $\hat{u}(\cdot ,%
\bar{\rho})\colon \lbrack \lambda _{1}(\bar{\rho}),+\infty )\rightarrow W$
is continuous.
\end{itemize}
\end{enumerate}
\end{theorem}

The ground state and principal solution of equation (\ref{eq:M}) can also be
expressed via the dual problem (cf. \cite{egnel, EgorKondr, ValIl2025, Qi,
Wei, ZhangM}), defined as follows:
\begin{equation}
\lambda _{1}^{\kappa }(\bar{\rho})=\sup \left\{ \lambda _{1}(\rho )\text{ }|%
\text{\ }\rho \in \mathbb{L}_{\mu }^{2},\ \Vert \bar{\rho}-\rho \Vert _{%
\mathbb{L}_{\mu }^{2}}^{2}=\kappa \right\} .  \label{eq-10}
\end{equation}

\begin{theorem}
\label{thm4} Assume that $0<\mu <\min \{N,4\}$, $\bar{\rho}\in \mathbb{L}%
_{\mu }^{2}$, and $\kappa >0$. Then the following statements hold:

\begin{enumerate}
\item[$(i)$] There exists a \emph{unique} maximizer $\check{\rho}\in \mathbb{%
L}_{\mu }^{2}$ of the dual inverse optimal problem (\ref{eq-10}). This
maximizer can be recovered pointwise almost everywhere via
\begin{equation*}
\check{\rho}(x,y)=\bar{\rho}(x,y)-\hat{u}(x)\hat{u}(y)\quad \text{a.e. in }%
\mathbb{R}^{N}\times \mathbb{R}^{N},  \label{eq:rho_recovery}
\end{equation*}%
where $\hat{u}\in W$ is a principal solution of equation (\ref{eq:M}) with $%
\lambda =\lambda _{1}^{\kappa }(\bar{\rho})$. Moreover, $\check{\rho}=\hat{%
\rho}(\lambda _{1}(\bar{\rho}),\bar{\rho})$ and
\begin{equation*}
\lambda _{1}^{\kappa }(\bar{\rho})=\lambda _{1}(\check{\rho})>\lambda _{1}(%
\bar{\rho}).
\end{equation*}

\item[$(ii)$] The principal solution $\hat{u}\in W$ of equation (\ref{eq:M})
with $\lambda =\lambda _{1}^{\kappa }(\bar{\rho})$ is uniquely determined
almost everywhere (up to its sign) by
\begin{equation*}
|\hat{u}(x)|=\sqrt{\bar{\rho}(x,x)-\check{\rho}(x,x)}\quad \text{a.e. in }%
\mathbb{R}^{N}.  \label{eq:principal_solution}
\end{equation*}

\item[$(iii)$] If $\bar{\rho}\geq 0$ almost everywhere in $\mathbb{R}%
^{N}\times \mathbb{R}^{N}$, then the ground state $\hat{u}\in W$ of equation
(\ref{eq:M}) with $\lambda =\lambda _{1}^{\kappa }(\bar{\rho})$ can be
recovered pointwise almost everywhere via
\begin{equation*}
\hat{u}(x)=\sqrt{\bar{\rho}(x,x)-\check{\rho}(x,x)}\quad \text{a.e. in }%
\mathbb{R}^{N}.  \label{eq:ground_state_recovery}
\end{equation*}
\end{enumerate}
\end{theorem}

This paper is organized as follows. After presenting some preliminaries in
Section 2, we prove Theorem \ref{thm1} in Section 3. In Section 4, we
investigate the existence and uniqueness of ground state solutions and prove
Theorem \ref{thm2}. Then, in Section 5, we prove Theorem \ref{thm3}. Section
6 is devoted to the proof of Theorem \ref{thm4}. Finally, we provide some
concluding remarks in Section 7.

\section{Preliminaries}

By the Sobolev theorem \cite{zhang} the following embeddings are continuous
and compact:

\begin{itemize}
\item $W \hookrightarrow L^{r}(\mathbb{R}^{N})$ for $2\leq r\leq\infty$, $%
N=1 $;

\item $W \hookrightarrow L^{r}(\mathbb{R}^{N})$ for $2\leq r<\infty$, $N=2$;

\item $W\hookrightarrow L^{r}(\mathbb{R}^{N})$ for $2\leq r<\frac{2N}{N-2}$,
$N\geq 3$.
\end{itemize}

In particular, via condition (\ref{cond:V}) there holds:
\begin{eqnarray*}
\Vert u\Vert _{L^{r}}\leq \mathcal{S}\Vert u\Vert _{W^{1,2}} &=&\mathcal{S}%
\left( \int_{\mathbb{R}^{N}}|\nabla u|^{2}dx+\int_{\mathbb{R}%
^{N}}|u|^{2}dx\right) \\
&\leq &\max \left\{ \mathcal{S},\frac{\mathcal{S}}{a_{V}}\right\} \left(
\int_{\mathbb{R}^{N}}|\nabla u|^{2}dx+\int_{\mathbb{R}^{N}}V(x)|u|^{2}dx%
\right) \\
&=&\max \left\{ \mathcal{S},\frac{\mathcal{S}}{a_{V}}\right\} \Vert u\Vert
_{W}.
\end{eqnarray*}

\begin{lemma}[Hardy-Littlewood-Sobolev inequality \protect\cite{SimonReed,
zhang}]
\label{L2-10} Assume that $1<a<b<+\infty $ is such that $\frac{1}{a}+\frac{1%
}{b}+\frac{\mu }{N}=2.$ Then there exists $\mathcal{C}_{HLS}>0$ such that
\begin{equation*}
\int_{\mathbb{R}^{3}}\int_{\mathbb{R}^{3}}\frac{|f(x)||g(y)|}{|x-y|^{\mu }}%
dxdy\leq \mathcal{C}_{HLS}\Vert f\Vert _{a}\Vert g\Vert _{b},\quad \forall
f\in L^{a}(\mathbb{R}^{N})\text{ and }\forall g\in L^{b}(\mathbb{R}^{N}).
\end{equation*}
\end{lemma}

\begin{lemma}
\label{lem:2} Assume that $0<\mu <\min \{N,4\}$, $\rho \in \mathbb{L}_{\mu
}^{2}$, and condition (\ref{cond:V}) is satisfied. Then

\begin{description}
\item[\textrm{(i)}] $-\infty <\lambda _{1}(\rho )<+\infty $ for all $\rho
\in \mathbb{L}_{\mu }^{2}$, where $\lambda _{1}(\rho )$ is a eigenvalue of $%
\mathcal{L}[\rho ]$, and there exists a normalized eigenfunction $\phi
_{1}[\rho ]\in W\setminus \{0\}$ such that
\begin{equation}
\mathcal{L}[\rho ]\phi _{1}[\rho ]=\lambda _{1}(\rho )\phi _{1}[\rho ]\text{
in }\mathbb{R}^{N}.  \label{eq:LinP}
\end{equation}

\item[\textrm{(ii)}] If $\rho \geq 0$ in $\mathbb{R}^{N}\times \mathbb{R}%
^{N} $, then $\lambda _{1}(\rho )$ is simple, and $\phi _{1}[\rho ]$ is
positive in $\mathbb{R}^{N}$.
\end{description}
\end{lemma}

\begin{proof}
By H\"{o}lder inequality, we have
\begin{eqnarray}
|(S[\rho ]\psi ,\psi )| &=&\left\vert \int_{\mathbb{R}^{N}}\int_{\mathbb{R}%
^{N}}\frac{\rho (x,y)\psi (x)\psi (y)}{|x-y|^{\mu }}dxdy\right\vert  \notag
\\
&\leq &\left( \int_{\mathbb{R}^{N}}\int_{\mathbb{R}^{N}}\frac{|\rho
(x,y)|^{2}}{|x-y|^{\mu }}dxdy\right) ^{1/2}\left( \int_{\mathbb{R}^{N}}\int_{%
\mathbb{R}^{N}}\frac{|\psi (x)|^{2}|\psi (y)|^{2}}{|x-y|^{\mu }}dxdy\right)
^{1/2}.  \label{eq:HLSIneq}
\end{eqnarray}%
Since $2<\frac{4N}{2N-\mu }<\frac{2N}{N-2}$ for $0<\mu <\min \{N,4\}$, by
Hardy-Littlewood-Sobolev inequality and condition (\ref{cond:V}), we get
\begin{equation*}
|(S[\rho ]\psi ,\psi )|\leq \mathcal{C}_{HLS}\Vert \rho \Vert _{\mathbb{L}%
_{\mu }^{2}}\Vert \psi \Vert _{L^{\frac{4N}{2N-\mu }}}^{2}\leq C_{a}\Vert
\rho \Vert _{\mathbb{L}_{\mu }^{2}}\Vert \psi \Vert _{W}^{2},\text{ }\forall
\psi \in W,  \label{eq:HLitEst}
\end{equation*}%
where $C_{a}=\mathcal{C}_{HLS}\max \{\mathcal{S},\frac{\mathcal{S}}{a_{V}}%
\}\in (0,+\infty )$ does not depend on $\psi \in W$ and $\rho \in \mathbb{L}%
_{\mu }^{2}$.

The operator $\mathcal{L}:=\mathcal{L}[\rho ]|_{\rho =0}:=-\Delta +V$,
defined under condition (\ref{cond:V}) with domain $D(\mathcal{L}):=W^{2,2}(%
\mathbb{R}^{N})$, is a self-adjoint operator (see, e.g., \cite{SimonReed}).
Furthermore,
\begin{equation}
\lambda _{1}(0):=\inf_{\psi \in C_{0}^{\infty }(\mathbb{R}^{N})\setminus
\{0\}}\frac{\int_{\mathbb{R}^{N}}|\nabla \psi |^{2}dx+\int_{\mathbb{R}%
^{N}}V|\psi |^{2}dx}{\int_{\mathbb{R}^{N}}|\psi (x)|^{2}dx}\in (-\infty
,+\infty ).  \label{eq:lambda0}
\end{equation}%
Hence, we derive
\begin{eqnarray*}
&&\inf_{\psi \in C_{0}^{\infty }(\mathbb{R}^{N})\setminus 0}\frac{\int_{%
\mathbb{R}^{N}}|\nabla \psi |^{2}dx+\int_{\mathbb{R}^{N}}V|\psi
|^{2}dx-C_{a}\Vert \rho \Vert _{\mathbb{L}_{\mu }^{2}}\Vert \psi \Vert
_{W}^{2}}{\int_{\mathbb{R}^{N}}|\psi (x)|^{2}dx} \\
&\leq &\inf_{\psi \in C_{0}^{\infty }(\mathbb{R}^{N})\setminus 0}\frac{(%
\mathcal{L}[\rho ]\psi ,\psi )}{\int_{\mathbb{R}^{N}}|\psi (x)|^{2}dx} \\
&\leq &\inf_{\psi \in C_{0}^{\infty }(\mathbb{R}^{N})\setminus 0}\frac{\int_{%
\mathbb{R}^{N}}|\nabla \psi |^{2}dx+\int_{\mathbb{R}^{N}}V|\psi
|^{2}dx+C_{a}\Vert \rho \Vert _{\mathbb{L}_{\mu }^{2}}\Vert \psi \Vert
_{W}^{2}}{\int_{\mathbb{R}^{N}}|\psi (x)|^{2}dx}.
\end{eqnarray*}%
By (\ref{eq:lambda0}), we have
\begin{eqnarray*}
&&\inf_{\psi \in C_{0}^{\infty }(\mathbb{R}^{N})\setminus 0}\frac{\int_{%
\mathbb{R}^{N}}|\nabla \psi |^{2}dx+\int_{\mathbb{R}^{N}}V|\psi
|^{2}dx-C_{a}\Vert \rho \Vert _{\mathbb{L}_{\mu }^{2}}\Vert \psi \Vert
_{W}^{2}}{\int_{\mathbb{R}^{N}}|\psi (x)|^{2}dx} \\
&=&\inf_{\psi \in C_{0}^{\infty }(\mathbb{R}^{N})\setminus 0}\frac{%
(1-C_{a}\Vert \rho \Vert _{\mathbb{L}_{\mu }^{2}})\left( \int_{\mathbb{R}%
^{N}}|\nabla \psi |^{2}dx+\int_{\mathbb{R}^{N}}V|\psi |^{2}dx\right) }{\int_{%
\mathbb{R}^{N}}|\psi (x)|^{2}dx}>-\infty .
\end{eqnarray*}%
Similarly,
\begin{equation*}
\inf_{\psi \in C_{0}^{\infty }(\mathbb{R}^{N})\setminus 0}\frac{\int_{%
\mathbb{R}^{N}}|\nabla \psi |^{2}dx+\int_{\mathbb{R}^{N}}V|\psi
|^{2}dx+C_{a}\Vert \rho \Vert _{\mathbb{L}_{\mu }^{2}}\Vert \psi \Vert
_{W}^{2}}{\int_{\mathbb{R}^{N}}|\psi (x)|^{2}dx}<+\infty .
\end{equation*}%
Thus,
\begin{equation}
-\infty <(1-C_{a}\Vert \rho \Vert _{\mathbb{L}_{\mu }^{2}})\lambda
_{1}(0)\leq \lambda _{1}(\rho )\leq (1+C_{a}\Vert \rho \Vert _{\mathbb{L}%
_{\mu }^{2}})\lambda _{1}(0)<+\infty .  \label{eq:lamlamlam}
\end{equation}

Let $\psi _{n}\in W$, $n=1,\ldots $ be a minimizing sequence of (\ref%
{eq:lambda1}). Clearly, $(\psi _{n})$ is bounded, and therefore by the
Banach-Alaoglu and the compact embedding $W\hookrightarrow L^{2}(\mathbb{R}%
^{N})$ there exists a nonzero limit point $\phi _{1}[\rho ]\in W\setminus 0$
in $W$. Hence, it follows that $\phi _{1}[\rho ]$ is a minimizer of (\ref%
{eq:lambda1}) and weakly satisfies (\ref{eq:LinP}). Since $S[\rho ]\phi
_{1}[\rho ]\in L^{2}(\mathbb{R}^{N})$, by the $L^{p}$-regularity results
from \cite{Giltrud}, we conclude that $\phi _{1}[\rho ]\in W_{loc}^{2,2}(%
\mathbb{R}^{N})$.

Now let us prove (ii). Since $\rho \geq 0$,
\begin{equation*}
\int_{\mathbb{R}^{N}}\int_{\mathbb{R}^{N}}\frac{\rho (x,y)\psi (x)\psi (y)}{%
|x-y|^{\mu }}dxdy\leq \int_{\mathbb{R}^{N}}\int_{\mathbb{R}^{N}}\frac{\rho
(x,y)|\psi (x)||\psi (y)|}{|x-y|^{\mu }}dxdy,\text{ }\forall \psi \in
\mathcal{L}_{\mu }^{2}.
\end{equation*}%
Hence,
\begin{equation*}
\frac{(\mathcal{L}[\rho ]|\phi _{1}[\rho ]|,|\phi _{1}[\rho ]|)}{(|\phi
_{1}[\rho ]|,|\phi _{1}[\rho ]|)}\leq \lambda _{1}(\rho )=\inf_{\psi \in
C_{0}^{\infty }(\mathbb{R}^{N})\setminus 0}\frac{(\mathcal{L}[\rho ]\psi
,\psi )}{(\psi ,\psi )}.
\end{equation*}%
Since $|\phi _{1}[\rho ]|\in W$, it follows that $|\phi _{1}[\rho ]|$ is
also a minimizer of (\ref{eq:lambda1}). Thus, without loss of generality, we
may assume that $\phi _{1}[\rho ]\geq 0$ a.e. in $\mathbb{R}^{N}$.

Given $\rho \geq 0$, this implies that $S[\rho ]\phi _{1}[\rho ]\geq 0$, and
consequently
\begin{equation*}
-\Delta \phi _{1}[\rho ]+V\phi _{1}[\rho ]\geq 0\text{ in }\mathbb{R}^{N}.
\end{equation*}%
Hence, by the Harnack inequality (see, e.g., Theorem 8.19 in \cite{Giltrud}%
), we derive that $\phi _{1}[\rho ]>0$ a.e. in $\mathbb{R}^{N}$. In
particular, this implies that $|\tilde{\phi}_{1}[\rho ]|>0$ a.e. in $\mathbb{%
R}^{N}$ for any eigenfunction $\tilde{\phi}_{1}[\rho ]$ corresponding to the
eigenvalue $\lambda _{1}(\rho )$. From this, it follows in a standard manner
(see also the proof of Proposition \ref{propGS} below) that the solution is
either positive or negative in $\mathbb{R}^{N}$, whence $\lambda _{1}(\rho )$
is simple.
\end{proof}

\begin{proposition}
\label{prop2} $\lambda_1(\rho)$ is a upper semicontinuous functional in $%
\mathbb{L}_{\mu}^{2}$.
\end{proposition}

\begin{proof}
Let $(\phi _{n})\subset C_{0}^{\infty }(\mathbb{R}^{N})$ be a countable
dense set in $C_{0}^{\infty }(\mathbb{R}^{N})$. Clearly,
\begin{equation*}
\lambda _{1}(\rho )=\inf_{n\geq 1}\frac{(\mathcal{L}[\rho ]\phi _{n},\phi
_{n})}{(\phi _{n},\phi _{n})},\text{ }\forall \rho \in \mathbb{L}_{\mu }^{2}.
\end{equation*}%
Hence for any $\tau \in \mathbb{R}$,
\begin{equation*}
\left\{ \rho \in \mathbb{L}_{\mu }^{2}:\lambda _{1}(\rho )<\tau \right\}
=\bigcup_{n=1}^{\infty }\left\{ \rho \in \mathbb{L}_{\mu }^{2}:\frac{(%
\mathcal{L}[\rho ]\phi _{n},\phi _{n})}{(\phi _{n},\phi _{n})}<\tau \right\}
.
\end{equation*}%
From (\ref{eq:HLSIneq}) it follows that ${\frac{(\mathcal{L}[\rho ]\phi
_{n},\phi _{n})}{(\phi _{n},\phi _{n})}}$, $\forall n=1,\ldots $, is
continuous with respect to $\rho \in \mathbb{L}_{\mu }^{2}$. Hence, the set
\begin{equation*}
\left\{ \rho \in \mathbb{L}_{\mu }^{2}:\frac{(\mathcal{L}[\rho ]\phi
_{n},\phi _{n})}{(\phi _{n},\phi _{n})}<\tau \right\} \text{ is open set in }%
\mathbb{L}_{\mu }^{2}
\end{equation*}%
for each $n\geq 1$, and therefore, the set $\left\{ \rho \in \mathbb{L}%
^{\gamma }:\lambda _{1}(\rho )<\tau \right\} $ is open for any $\tau \in
\mathbb{R}$. This means that $\lambda _{1}(\rho )$ is a upper semicontinuous
functional in $\mathbb{L}_{\mu }^{2}$.
\end{proof}

\begin{proposition}
\label{prop23} \label{prop1} $\lambda _{1}(\rho )$ is a concave functional
in $\mathbb{L}_{\mu }^{2}$, that is
\begin{equation*}
\lambda _{1}(t\rho _{1}+(1-t)\rho _{2})\geq t\lambda _{1}(\rho
_{1})+(1-t)\lambda _{1}(\rho _{2}),\text{ }\forall \rho _{1},\rho _{2}\in
\mathbb{L}_{\mu }^{2}\text{ and }\forall t\in \lbrack 0,1].
\end{equation*}
\end{proposition}

\begin{proof}
Let $\rho _{1},\rho _{2}\in \mathbb{L}_{\mu }^{2}$. Observe,
\begin{eqnarray*}
&&\lambda _{1}(t\rho _{1}+(1-t)\rho _{2}) \\
&=&\inf_{\{\psi \in C_{0}^{\infty }(\mathbb{R}^{N})\setminus \{0\}:\Vert
\psi \Vert _{L^{2}}=1\}}(\mathcal{L}[t\rho _{1}+(1-t)\rho _{2}]\psi ,\psi )
\\
&\geq &t\inf_{\{\psi \in C_{0}^{\infty }(\mathbb{R}^{N})\setminus
\{0\}:\Vert \psi \Vert _{L^{2}}=1\}}(\mathcal{L}[\rho _{1}]\psi ,\psi
)+(1-t)\inf_{\{\psi \in C_{0}^{\infty }(\mathbb{R}^{N})\setminus \{0\}:\Vert
\psi \Vert _{L^{2}}=1\}}(\mathcal{L}[\rho _{2}]\psi ,\psi ) \\
&=&t\lambda _{1}(\rho _{1})+(1-t)\lambda (\rho _{2}),\text{ }\forall t\in
\lbrack 0,1].
\end{eqnarray*}%
The proof is complete.
\end{proof}

Define
\begin{equation*}
M_{\lambda }:=\left\{ \rho \in \mathbb{L}_{\mu }^{2}:\lambda \leq \lambda
_{1}(\rho )\right\}
\end{equation*}

\begin{corollary}
The set $M_{\lambda}$ is strictly convex.
\end{corollary}

\begin{proposition}
\label{prop3} Assume that $\rho _{n}\in \mathbb{L}_{\mu }^{2}$, $n=1,\ldots $%
, and that $\rho _{n}\rightarrow \bar{\rho}$ in $\mathbb{L}_{\mu }^{2}$ as $%
n\rightarrow +\infty $. Then there exists a subsequence $(n_{i})_{i=1}^{%
\infty }$ with $n_{i}\rightarrow +\infty $ as $i\rightarrow +\infty $, and
corresponding normalized eigenfunctions $\phi _{1}[\rho _{n_{i}}]$ of $%
\mathcal{L}[\rho _{n_{i}}]$, such that
\begin{eqnarray*}
&&\phi _{1}[\rho _{n_{i}}]\rightarrow \phi _{1}[\bar{\rho}]\text{ strongly
in }W\text{ as }i\rightarrow +\infty \\
&&\lambda _{1}(\rho _{n_{i}})\rightarrow \lambda _{1}(\bar{\rho})\text{ as }%
i\rightarrow +\infty .
\end{eqnarray*}
\end{proposition}

\begin{proof}
Using Proposition \ref{prop2} and by (\ref{eq:lamlamlam}) we derive
\begin{equation*}
-\infty <\limsup_{n\rightarrow +\infty }\lambda _{1}(\rho _{n})\leq \lambda
_{1}(\rho _{0})<+\infty .
\end{equation*}%
From this, we infer that, without loss of generality, we may assume there
exists a limit
\begin{equation*}
\lim_{n\rightarrow +\infty }\lambda _{1}(\rho _{n})=\hat{\lambda}>-\infty .
\end{equation*}%
Since $\Vert \phi _{1}[\rho _{n}]\Vert _{L^{2}}=1$, $n=1,\ldots $, this
implies that $\phi _{1}[\rho _{n}]$ is bounded in $W$. Thus Banach-Alaoglu
and the compact embedding $W\hookrightarrow L^{2}(\mathbb{R}^{N})$ yields
that there exists a subsequence $(n_{i})_{i=1}^{\infty }$ with $%
n_{i}\rightarrow +\infty $ as $i\rightarrow +\infty $ such that $\phi
_{1}[\rho _{n_{i}}]\rightharpoonup \hat{\phi}$ weakly in $W$ and strongly $%
\phi _{1}[\rho _{n_{i}}]\rightarrow \hat{\phi}$ in $L^{2}(\mathbb{R}^{N})$
for some $\hat{\phi}\in W$. Since $\Vert \phi _{1}[\rho _{n_{i}}]\Vert
_{L^{2}}=1,i=1,\ldots $, we get that $\Vert \hat{\phi}\Vert _{L^{2}}=1$ and
therefore, $\hat{\phi}\neq 0$. Hence, by the weak lower semicontinuity of
the norm $\Vert \cdot \Vert _{W}$, we have
\begin{equation*}
(\mathcal{L}[\rho _{0}]\hat{\phi},\hat{\phi})\leq \liminf_{i\rightarrow
+\infty }(\mathcal{L}[\rho _{0}]\phi _{1}[\rho _{n_{i}}],\phi _{1}[\rho
_{n_{i}}])=\hat{\lambda}\leq \lambda _{1}(\rho _{0}).
\end{equation*}%
This implies that $(\mathcal{L}[\rho _{0}]\hat{\phi},\hat{\phi})=\lambda
_{1}(\rho _{0})$, $\hat{\phi}=\phi _{1}[\rho _{0}]$, and $\phi _{1}[\rho
_{n_{i}}]\rightarrow \phi _{1}[\rho _{0}]$ strongly in $W$ as $i\rightarrow
+\infty $.
\end{proof}

\begin{lemma}
\label{lem10} The functional $\lambda _{1}(\rho )$ is Gateaux differentiable
at any point $\rho \in \mathbb{L}_{\mu }^{2}$, moreover
\begin{equation}
D_{\rho }\lambda _{1}(\rho )(h)=-\frac{1}{\Vert \phi _{1}[\rho ]\Vert
_{L^{2}}^{2}}\int_{\mathbb{R}^{N}}\int_{\mathbb{R}^{N}}\frac{\phi _{1}[\rho
](x)\phi _{1}[\rho ](y)}{|x-y|^{\mu }}h(x,y)dxdy,\text{ }\forall h\in
\mathbb{L}_{\mu }^{2}.  \label{eq:Val}
\end{equation}
\end{lemma}

\begin{proof}
Let $\varepsilon >0$, $h \in \mathbb{L}_{\mu }^{2}$. By Proposition \ref%
{prop1} the map $[0,1)\ni \varepsilon \mapsto \lambda_1(\rho+\varepsilon h)$
is a concave function, and therefore, it is continuous and admits
directional derivative $\partial_+\lambda_1(\rho)(h)$.

Since $\phi _{1}[\rho ]$ is a minimizer of (\ref{eq:lambda1}), we have
\begin{eqnarray*}
\lambda _{1}(\rho +\varepsilon h) &=&\inf_{\psi \in C_{0}^{\infty }\setminus
0}\frac{1}{\|\psi\|_{L^{2}}^2 }\left( (\mathcal{L}[\rho ]\psi ,\psi
)-\varepsilon \int_{\mathbb{R}^{N}}\int_{\mathbb{R}^{N}}\frac{\psi (x)\psi
(y)}{|x-y|^{\mu }}h(x,y)dxdy\right) \\
&\leq &\lambda _{1}(\rho )-\varepsilon \frac{1}{\|\phi_{1}[\rho]\|_{L^{2}}^2
}\int_{\mathbb{R}^{N}}\int_{\mathbb{R}^{N}}\frac{\phi _{1}[\rho ](x)\phi
_{1}[\rho ](y)}{|x-y|^{\mu }}h(x,y)dxdy,
\end{eqnarray*}%
and thus,
\begin{equation}
\partial_+\lambda_1(\rho)(h)\leq -\int_{\mathbb{R}^{N}}\int_{\mathbb{R}^{N}}%
\frac{\phi _{1}[\rho ](x)\phi _{1}[\rho ](y)}{|x-y|^{\mu }}h(x,y)dxdy.
\label{e2-6}
\end{equation}

Since $\mathbb{R}\ni \varepsilon \mapsto \lambda _{1}(\rho +\varepsilon h)$
is concave function, we have
\begin{equation*}
\partial _{+}\lambda _{1}(\rho )(h)=\sup_{\varepsilon \in \mathbb{R}}\frac{%
\lambda _{1}(\rho +\varepsilon h)-\lambda _{1}(\rho )}{\varepsilon },
\end{equation*}%
and therefore
\begin{equation*}
\partial _{+}\lambda _{1}(\rho )(h)\geq \frac{\lambda _{1}(\rho +\varepsilon
h)-\lambda _{1}(\rho )}{\varepsilon },\text{ }\forall \varepsilon \in
(-a_{h},a_{h}).
\end{equation*}%
Observe
\begin{eqnarray*}
&&\lambda _{1}(\rho +\varepsilon h)=(\mathcal{L}[\rho +\varepsilon h]\phi
_{1}[\rho +\varepsilon h],\phi _{1}[\rho +\varepsilon h]),\text{ }\forall
\varepsilon \in \mathbb{R}, \\
&&\lambda _{1}(\rho )\leq (\mathcal{L}[\rho ]\phi _{1}[\rho +\varepsilon
h],\phi _{1}[\rho +\varepsilon h]),\text{ }\forall \varepsilon \in \mathbb{R}%
.
\end{eqnarray*}%
Hence,
\begin{equation*}
\partial _{+}\lambda _{1}(\rho )(h)\geq \frac{(\mathcal{L}[\rho +\varepsilon
h]\phi _{1}[\rho +\varepsilon h],\phi _{1}[\rho +\varepsilon h])-(\mathcal{L}%
[\rho ]\phi _{1}[\rho +\varepsilon h],\phi _{1}[\rho +\varepsilon h])}{%
\varepsilon },
\end{equation*}%
for $\varepsilon >0$, and therefore,
\begin{equation*}
\partial _{+}\lambda _{1}(\rho )(h)\geq -\int_{\mathbb{R}^{N}}\int_{\mathbb{R%
}^{N}}\frac{\phi _{1}[\rho +\varepsilon h](x)\phi _{1}[\rho +\varepsilon
h](y)}{|x-y|^{\mu }}h(x,y)dxdy,\text{ }\forall \varepsilon \in
(-a_{h},a_{h}).
\end{equation*}

By Proposition \ref{prop3}, $\phi _{1}[\rho +\varepsilon _{j}h]\rightarrow
\phi _{1}[\rho ]$ in $W$ as $j\rightarrow +\infty $ for some sequence $%
(\varepsilon _{j})$ such that $\lim_{j\rightarrow +\infty }\varepsilon
_{j}=0 $. Thus we obtain
\begin{equation*}
\partial _{+}\lambda _{1}(\rho )(h)\geq -\int_{\mathbb{R}^{N}}\int_{\mathbb{R%
}^{N}}\frac{\phi _{1}[\rho ](x)\phi _{1}[\rho ](y)}{|x-y|^{\mu }}h(x,y)dxdy.
\end{equation*}%
Combining this with (\ref{e2-6}), we obtain
\begin{equation*}
\partial _{+}\lambda _{1}(\rho )(h)=-\int_{\mathbb{R}^{N}}\int_{\mathbb{R}%
^{N}}\frac{\phi _{1}[\rho ](x)\phi _{1}[\rho ](y)}{|x-y|^{\mu }}h(x,y)dxdy.
\end{equation*}%
Since this equality holds for any $h\in \mathbb{L}_{\mu }^{2}$, it follows
that $\lambda _{1}(\rho )$ is Gateaux differentiable at $\rho \in \mathbb{L}%
_{\mu }^{2}$, and (\ref{eq:Val}) is valid.
\end{proof}

\section{Proof of Theorem~\protect\ref{thm1}}

\begin{lemma}
\label{lem1} For any $\bar{\rho}\in \mathbb{L}_{\mu }^{2}$ and $\lambda
>\lambda _{1}(\bar{\rho})$ the minimization problem (\ref{eq:p}) admits a
unique solution $\hat{\rho}(\bar{\rho})\in \mathbb{L}_{\mu }^{2}\setminus
\{0\}$.
\end{lemma}

\begin{proof}
Let $\bar{\rho}\in \mathbb{L}_{\mu }^{2}$ and $\lambda >\lambda _{1}(\bar{%
\rho})$. Define
\begin{equation*}
\tilde{M}_{\lambda }:=\left\{ \rho \in \mathbb{L}_{\mu }^{2}:\lambda
_{1}(\rho )\geq \lambda \right\} .
\end{equation*}%
Consider the constrained minimization problem
\begin{equation}
\hat{\mathcal{P}}(\lambda ,\bar{\rho})=\min \left\{ \Vert \rho -\bar{\rho}%
\Vert _{\mathbb{L}_{\mu }^{2}}^{2}:\rho \in \tilde{M}_{\lambda }\right\} .
\label{MinP}
\end{equation}%
Below we will see that minimization problems (\ref{eq:p}) and (\ref{MinP})
are equivalent. Observe $\tilde{M}_{\lambda }\neq \emptyset $. Indeed, take $%
\rho \in \mathbb{L}_{\mu }^{2}$ such that $\lambda _{1}(\rho )>0$. Then for
any $a\geq \lambda $, we have $a=\lambda _{1}\left( \frac{a\cdot \rho }{%
\lambda _{1}(\rho )}\right) \geq \lambda $, and thus, $\frac{\lambda \cdot
\rho }{\lambda _{1}(\rho )}\in \tilde{M}_{\lambda }$.

Let $\{\rho _{n}\}\subset \tilde{M}_{\lambda}$ be a minimizing sequence of (%
\ref{MinP}), i.e., $\Vert \rho _{n}-\bar{\rho} \Vert _{\mathbb{L}_{\mu
}^{2}}\rightarrow \hat{\mathcal{P}}(\lambda ,\rho )$ as $n\rightarrow
+\infty $. Then $\Vert \rho _{n}\Vert _{\mathbb{L}_{\mu }^{2}}$ is bounded,
and therefore, the reflexivity of $\mathbb{L}_{\mu }^{2}$ yields that there
exists a subsequence, denoting again by $\{\rho _{n}\}$, such that $\rho
_{n}\rightharpoonup \hat{\rho}(\bar{\rho})$ weakly in $\mathbb{L}_{\mu }^{2}$
for some $\hat{\rho}(\bar{\rho})\in \mathbb{L}_{\mu }^{2}$.

By Propositions \ref{prop2} and \ref{prop1}, the functional $\lambda
_{1}(\rho )$ is concave and upper semicontinuous on $\mathbb{L}_{\mu }^{2}$,
and therefore, it is weakly upper semicontinuous. Hence,
\begin{equation*}
\lambda \leq \limsup_{n\rightarrow +\infty }\lambda _{1}(\rho _{n})\leq
\lambda _{1}(\hat{\rho}(\bar{\rho})),
\end{equation*}%
and therefore, $\hat{\rho}(\bar{\rho})\neq 0$. Thus, $\hat{\rho}(\bar{\rho}%
)\in \tilde{M}_{\lambda }$. On the other hand, by the weak lower
semi-continuity of $\Vert \rho -\rho \Vert _{\mathbb{L}_{\mu }^{2}}^{2}$ we
have
\begin{equation*}
\Vert \hat{\rho}(\bar{\rho})-\rho \Vert _{\mathbb{L}_{\mu }^{2}}^{2}\leq
\liminf_{n\rightarrow +\infty }\Vert \rho _{n}-\rho \Vert _{\mathbb{L}_{\mu
}^{2}}^{2}=\hat{\mathcal{P}}(\lambda ,\rho ).
\end{equation*}%
Since $\hat{\rho}(\bar{\rho})\in \tilde{M}_{\lambda }$, this implies that $%
\Vert \hat{\rho}(\bar{\rho})-\rho \Vert _{\mathbb{L}_{\mu }^{2}}^{2}=\hat{%
\mathcal{P}}(\lambda ,\rho )$, that is $\hat{\rho}(\bar{\rho})$ is a
minimizer of (\ref{MinP}). Recalling that $\lambda _{1}(\hat{\rho}(\bar{\rho}%
))\geq \lambda >\lambda _{1}(\bar{\rho})$ we conclude that $\hat{\rho}\neq
\bar{\rho}$. Since $\Vert \cdot -\bar{\rho}\Vert _{\mathbb{L}_{\mu
}^{2}}^{2} $ is a convex functional and $\tilde{M}_{\lambda }$ is a strictly
convex set, $\hat{\rho}(\bar{\rho})$ is the unique minimizer. Furthermore, $%
\hat{\rho}(\bar{\rho})\in \partial \tilde{M}_{\lambda }=M_{\lambda }=\{\rho
\in \mathbb{L}_{\mu }^{2}:\lambda =\lambda _{1}(\rho )\}$, and thus,
minimization problems (\ref{eq:p}) and (\ref{MinP}) are equivalent.
\end{proof}

\begin{lemma}
\label{lem2} Assume $0<\mu <\min \{N,4\}$, $\bar{\rho}\in \mathbb{L}_{\mu
}^{2}$, and $\lambda >\lambda _{1}[\rho ]$. Then there exists a weak
solution $\hat{u}:=\hat{u}(\lambda ,\bar{\rho})\in W$ of equation (\ref{eq:M}%
) such that
\begin{equation*}
\hat{\rho}(x,y)=\bar{\rho}(x,y)-\hat{u}(x)\hat{u}(y)\quad \text{a.e. in }%
\mathbb{R}^{N}\times \mathbb{R}^{N}.
\end{equation*}%
where $\hat{\rho}:=\hat{\rho}(\bar{\rho})$ is a minimizer of inverse optimal
problem (\ref{eq:p}). Moreover, $\lambda =\lambda _{1}[\hat{\rho}]$, and
\begin{equation*}
\phi _{1}[\hat{\rho}]=\frac{\hat{u}(\lambda ,\bar{\rho})}{\Vert \hat{u}%
(\lambda ,\bar{\rho})\Vert _{L^{2}}}.
\end{equation*}
\end{lemma}

\begin{proof}
By Lemma~\ref{lem1}, there exists $\hat{\rho}\in \mathbb{L}_{\mu
}^{2}\setminus \{0\}$ solving minimization problem (\ref{MinP}). Due to the
convexity of the functionals $\Vert (\cdot )-\bar{\rho}\Vert _{2}^{2}$ and $%
(-\lambda _{1}(\cdot ))$, the Kuhn-Tucker Theorem (see, e.g., Zeidler \cite%
{Zeidler}, Theorem 47.E, p. 394) yields that there exist $\mu _{1},\mu
_{2}\in \mathbb{R}$, $|\mu _{1}|+|\mu _{2}|\neq 0$ such that the Lagrange
function
\begin{equation*}
\Lambda (\rho ):=\mu _{1}\Vert \rho -\bar{\rho}\Vert _{2}^{2}+\mu
_{2}(\lambda -\lambda _{1}(\rho )),
\end{equation*}%
satisfies the following conditions:
\begin{eqnarray}
&&\Lambda (\hat{\rho})=\min_{\rho \in \mathbb{L}_{\mu }^{2}}\Lambda (\rho )
\label{eq1} \notag\\
&&\mu _{2}(\lambda -\lambda _{1}(\hat{\rho}))=0,  \label{eq2} \\
&&\mu _{1},\mu _{2}\geq 0. \notag 
\end{eqnarray}%
Using Lemma \ref{lem10} we conclude that $\Lambda (\rho )$ is Gateaux
differentiable at any point $\rho \in \mathbb{L}_{\mu }^{2}$. Hence
\begin{equation}
2\mu _{1}\int (\hat{\rho}-\bar{\rho})h\,dx-\mu _{2}D_{\rho }\lambda _{1}(%
\hat{\rho})(h)=0,\text{ }\forall h\in C_{0}^{\infty },\text{ }h\geq 0.
\label{eq:ner}
\end{equation}%
Suppose that $\mu _{1}=0$. Then $\mu _{2}>0$. It is evidently that there
exists $h_{0}\in C_{0}^{\infty }$ such that $D_{\rho }\lambda _{1}(\hat{\rho}%
)(h_{0})>\neq 0$. Thus, by (\ref{eq:ner}) we obtain a contradiction
\begin{equation*}
-\mu _{2}D_{\rho }\lambda _{1}(\hat{\rho})(h_{0})=0.
\end{equation*}%
In the same manner, it is derived that $\mu _{2}\neq 0$, which by (\ref{eq2}%
) implies that $\lambda _{1}(\hat{\rho})=\lambda $. Hence, we may assume
that $\mu _{1}=1$. From (\ref{eq:ner}) and (\ref{eq:Val}) we have
\begin{equation}
\hat{\rho}(x,y)=\bar{\rho}(x,y)-\varpi \phi _{1}[\hat{\rho}](x)\phi _{1}[%
\hat{\rho}](y),  \label{eq:rhrho}
\end{equation}%
where $\varpi :=\frac{\mu _{2}}{2}>0$. Since
\begin{equation*}
-\Delta \phi _{1}[\hat{\rho}]+V(x)\phi _{1}[\hat{\rho}]-\int_{\mathbb{R}^{N}}%
\frac{\hat{\rho}(x,y)\phi _{1}[\hat{\rho}](y)}{|x-y|^{\mu }}dy=\lambda _{1}(%
\hat{\rho})\phi _{1}[\hat{\rho}],
\end{equation*}%
we derive by (\ref{eq:rhrho}) that
\begin{equation*}
-\Delta \phi _{1}[\hat{\rho}]+V(x)\phi _{1}[\hat{\rho}]-\int_{\mathbb{R}^{N}}%
\frac{{\bar{\rho}}(x,y)\phi _{1}[\hat{\rho}](y)}{|x-y|^{\mu }}dy+\varpi
\int_{\mathbb{R}^{N}}\frac{\phi _{1}[\hat{\rho}](y)\phi _{1}[\hat{\rho}](y)}{%
|x-y|^{\mu }}dy\phi _{1}[\hat{\rho}](x)=\lambda \phi _{1}[\hat{\rho}].
\end{equation*}%
Define $\hat{u}:=\hat{u}(\lambda ,\bar{\rho})=\varpi ^{1/2}\phi _{1}[\hat{%
\rho}]$. Then
\begin{equation*}
-\Delta \hat{u}+V(x)\hat{u}-\int_{\mathbb{R}^{N}}\frac{{\bar{\rho}}(x,y)\hat{%
u}(y)}{|x-y|^{\mu }}dy+\int_{\mathbb{R}^{N}}\frac{\hat{u}^{2}(y)}{|x-y|^{\mu
}}dy\hat{u}(x)=\lambda \hat{u},
\end{equation*}%
and
\begin{equation}
\hat{\rho}(x,y)=\bar{\rho}(x,y)-\hat{u}(x)\hat{u}(y).  \label{eq:39}
\end{equation}
\end{proof}

Observe, (\ref{eq:39}) implies
\begin{equation*}
\hat{u}^{2}(x)=\bar{\rho}(x,x)-\hat{\rho}(x,x)\quad \text{a.e. in }\mathbb{R}%
^{N}.
\end{equation*}%
Hence, $\bar{\rho}(x,x)-\hat{\rho}(x,x)\geq 0$ a.e. in $\mathbb{R}^{N}$ and
therefore
\begin{equation*}
|\hat{u}(x)|=\sqrt{\bar{\rho}(x,x)-\hat{\rho}(x,x)}\quad \text{a.e. in }%
\mathbb{R}^{N}.
\end{equation*}%
This completes the proof of Theorem~\ref{thm1}.

\section{Proof of Theorem \protect\ref{thm2}}

To prove that $\hat{u}$ is a ground state, it is sufficient to show that
equation (\ref{eq:M}) admits a \emph{nonnegative} ground state $\bar{u}\in W$%
. Once this is established, the uniqueness of the nonnegative solution in
the cone of non-negative solutions implies that the unique nonnegative
solution must coincide with this ground state. Hence $\hat{u}$, being the
unique nonnegative solution, is itself a ground state. The existence of a
nonnegative ground state is given by the following proposition.

\begin{proposition}
\label{propGS} Assume $0<\mu <\min \{N,4\}$. Let $\rho \in \mathbb{L}_{\mu
}^{2}(\mathbb{R}^{N})$ and $\lambda >\lambda _{1}(\rho )$. Then equation (%
\ref{eq:M}) possesses a ground state $\hat{u}_{g}\in W$. Moreover, $\hat{u}%
_{g}\in W$ is a stable solution.

If $\rho \geq 0$ a.e. in $\mathbb{R}^N\times \mathbb{R}^N$, the ground state
$\hat{u}_g$ is either positive or negative in $\mathbb{R}^{N}$.
\end{proposition}

\begin{proof}
Let us show that $E_{\lambda }(u)$ is a coercive functional on $W$. Suppose $%
\Vert u_{n}\Vert _{W}\rightarrow +\infty $ as $n\rightarrow +\infty $.
Denote
\begin{equation*}
t_{n}=\Vert u_{n}\Vert _{W},\quad v_{n}=\frac{u_{n}}{\Vert u_{n}\Vert _{W}}%
,\quad n=1,2,\dots .
\end{equation*}%
Since $\Vert v_{n}\Vert _{W}=1$, $n=1,\ldots $, by the Banach-Alaoglu and
Sobolev theorems\ there exists a subsequence, still denoted $(v_{n})$, and
some $v\in W$ such that $v_{n}\rightharpoonup v$ weakly in $W$ and $%
v_{n}\rightarrow v$ strongly in $L^{2}(\mathbb{R}^{N}).$ If $v\not\equiv 0$,
the nonlinear term dominates and $E_{\lambda }(u_{n})\rightarrow +\infty $
as $n\rightarrow +\infty $. If $v=0$, then
\begin{equation*}
\frac{E_{\lambda }(t_{n}v_{n})}{\frac{t_{n}^{2}}{2}\int_{\mathbb{R}%
^{N}}|\nabla v_{n}|^{2}dx}\rightarrow 1\quad \text{as }n\rightarrow \infty
\end{equation*}%
implies again that $E_{\lambda }(u_{n})\rightarrow +\infty $ as $%
n\rightarrow +\infty $. Thus, $E_{\lambda }(u)$ is coercive on $W$.

Consider
\begin{equation*}
\hat{E}_{\lambda }[\rho ]:=\min_{u\in W}E_{\lambda }(u)  \label{eq:gs}
\end{equation*}%
Let $(u_{n})\subset W$ be a minimizing sequence of this problem. Since $%
E_{\lambda }(u)$ is coercive on $W$, the sequence is bounded in $W$. Hence
by the Banach-Alaoglu and Sobolev theorems\ there exists a subsequence,
still denoted $(u_{n})$, and some $\hat{u}_{g}\in W$ such that
\begin{equation*}
u_{n}\rightharpoonup \hat{u}_{g}\quad \text{weakly in }W\quad \text{and}%
\quad u_{n}\rightarrow \hat{u}_{g}\quad \text{strongly in }L^{2}(\mathbb{R}%
^{N}).
\end{equation*}%
Since $0<\mu <\min \{N,4\}$, $E_{\lambda }$ is weakly lower semicontinuous,
and therefore
\begin{equation*}
E_{\lambda }(\hat{u}_{g})\leq \liminf_{n\rightarrow \infty }E_{\lambda
}(u_{n})=\hat{E}_{\lambda }[\rho ].
\end{equation*}%
It is evident that $\hat{E}_{\lambda }[\rho ]<0$ for all $\lambda >\lambda
_{1}(\rho )$. Therefore, $\hat{u}_{g}\not\equiv 0$. Consequently, $%
E_{\lambda }(\hat{u}_{g})=\hat{E}_{\lambda }[\rho ]$, which implies that $%
D_{u}E_{\lambda }(\hat{u}_{g})=0$. Since $\hat{u}_{g}$ is a minimizer of $%
E_{\lambda }(u)$ in $W$, the second derivative of the functional is
non-negative at $\hat{u}_{g}:D_{uu}E_{\lambda }(\hat{u}_{g})(\psi ,\psi
)\geq 0\quad \forall \psi \in W.$ Thus, $\hat{u}_{g}$ is a \emph{stable
ground state}.

Let $\rho \geq 0$ a.e. in $\mathbb{R}^{N}\times \mathbb{R}^{N}$. Suppose,
contrary to our claim that
\begin{equation*}
\hat{u}_{g}^{+}(x)=\max \{u(x),0\}\not\equiv 0,\text{ }\hat{u}%
_{g}^{-}(x)=\min \{u(x),0\}\not\equiv 0\text{ in }\mathbb{R}^{N}.
\end{equation*}%
Hence
\begin{eqnarray*}
\hat{E}_{\lambda }[\rho ] &=&E_{\lambda }(\hat{u}_{g}^{+}+\hat{u}_{g}^{-}) \\
&=&E_{\lambda }(\hat{u}_{g}^{+})+E_{\lambda }(\hat{u}_{g}^{-})-\int_{\mathbb{%
R}^{N}}\frac{\rho (x,y)\hat{u}_{g}^{+}(x)\hat{u}_{g}^{-}(y)}{|x-y|^{\mu }}%
dxdy+\frac{1}{2}\int_{\mathbb{R}^{N}}(|x|^{-\mu }\ast |\hat{u}_{g}^{+}|^{2})|%
\hat{u}_{g}^{-}|^{2}dx \\
&\geq &E_{\lambda }(\hat{u}_{g}^{+})+E_{\lambda }(\hat{u}_{g}^{-})\geq 2\hat{%
E}_{\lambda }[\rho ].
\end{eqnarray*}%
We get a contradiction. Hence, $\hat{u}_{g}>0$ or $\hat{u}_{g}<0$ a.e. in $%
\mathbb{R}^{N}$. Since
\begin{equation*}
E_{\lambda }(\hat{u}_{g})=\min_{u\in W}E_{\lambda }(u),
\end{equation*}%
$D_{uu}E_{\lambda }(\hat{u}_{g})(h,h)\geq 0$ for any $h\in W$, whence $\hat{u%
}_{g}$ is a stable solution.
\end{proof}

\begin{proposition}
\label{Prop:4} Assume $\rho \geq 0$ in $\mathbb{R}^N\times \mathbb{R}^N$.
Let $\tilde{u} \in W$ be a nonnegative weak solution of equation~\eqref{eq:M}%
. Then the function
\begin{equation*}
\tilde{\rho}(x,y) = \rho(x,y) - \tilde{u}(x)\tilde{u}(y)
\end{equation*}
is a local minimizer of the functional
\begin{equation*}
\mathcal{P}(\theta) = \|\theta - \rho\|_{\mathbb{L}_{\mu}^{2}}^{2}
\end{equation*}
over the manifold $M_{\lambda}$.
\end{proposition}

\begin{proof}
Let $\tilde{u}\in W$ be a nonnegative weak solution of equation (\ref{eq:M}%
). Then
\begin{equation*}
(\mathcal{L}[\rho ]\tilde{u},\tilde{u})+\iint_{\mathbb{R}^{N}\times \mathbb{R%
}^{N}}\frac{\tilde{u}^{2}(x)\tilde{u}^{2}(y)}{|x-y|^{\mu }}dxdy=0.
\end{equation*}%
This implies that the function $(x,y)\mapsto \tilde{u}(x)\tilde{u}(y)$
belongs to $\mathbb{L}_{\mu }^{2}$, and therefore,
\begin{equation*}
\tilde{\rho}(x,y)=\rho (x,y)-\tilde{u}(x)\tilde{u}(y),\quad \forall (x,y)\in
\mathbb{R}^{N}\times \mathbb{R}^{N}.
\end{equation*}%
belongs to $\mathbb{L}_{\mu }^{2}$, and $\tilde{\rho}\geq 0$ in $\mathbb{R}%
^{N}\times \mathbb{R}^{N}$. Hence, by Lemma \ref{lem:2}, the operator
\begin{equation*}
\mathcal{L}[\tilde{\rho}]:=-\Delta +V-S[\tilde{\rho}]
\end{equation*}%
has a simple principal eigenvalue $\lambda _{1}(\tilde{\rho})\in (0,+\infty
) $, with a unique normalized eigenfunction $\phi _{1}[\tilde{\rho}]\in
W\setminus \{0\}$ satisfying
\begin{equation}
\mathcal{L}[\tilde{\rho}]\phi _{1}[\tilde{\rho}]=\lambda _{1}(\tilde{\rho}%
)\phi _{1}[\tilde{\rho}]  \label{eq:eigentild}
\end{equation}%
in $\mathbb{R}^{N}$, and $\phi _{1}[\tilde{\rho}]>0$ a.e. in $\mathbb{R}^{N}$%
.

On the other hand, since $\tilde{u}\in W$ is a nonnegative weak solution of (%
\ref{eq:M}), we have
\begin{equation*}
\mathcal{L}[\tilde{\rho}]\tilde{u}=\lambda \tilde{u}.
\end{equation*}%
Testing this equation with $\phi _{1}[\tilde{\rho}]$, testing (\ref%
{eq:eigentild}) with $\tilde{u}$, integrating by parts, and subtracting
yields
\begin{equation*}
(\lambda -\lambda _{1}(\tilde{\rho}))\int_{\mathbb{R}^{N}}\tilde{u}(x)\phi
_{1}[\tilde{\rho}](x)dx=0.
\end{equation*}%
Since $\tilde{u}\geq 0$ and $\phi _{1}[\tilde{\rho}]>0$ a.e. in $\mathbb{R}%
^{N}$, it follows that
\begin{equation*}
\lambda =\lambda _{1}(\tilde{\rho}).
\end{equation*}%
Thus, $\tilde{\rho}\in \partial M_{\lambda }$, and we may assume (up to
normalization) that $\tilde{u}=\phi _{1}[\tilde{\rho}]$.

By Lemma \ref{lem10} the functional $\lambda _{1}(\tilde{\rho})$ admits the
directional derivative, and since $\tilde{u}=\phi _{1}[\tilde{\rho}]$,
\begin{equation}
D_{\rho }\lambda _{1}(\tilde{\rho})(h)=-\frac{1}{\Vert \tilde{u}\Vert
_{L^{2}}^{2}}\int_{\mathbb{R}^{N}}\int_{\mathbb{R}^{N}}\frac{\tilde{u}(x)%
\tilde{u}(y)}{|x-y|^{\mu }}h(x,y)dxdy,\text{ }\forall h\in \mathbb{L}_{\mu
}^{2}.  \label{eq:dirtil}
\end{equation}%
By Proposition~\ref{prop1}, the map
\begin{equation*}
\lbrack 0,1)\ni t\mapsto \lambda _{1}\left( t(\tilde{\rho}+h)+(1-t)\tilde{%
\rho}\right)
\end{equation*}%
is concave. Therefore,
\begin{equation*}
\lambda _{1}(\tilde{\rho}+th)-\lambda _{1}(\tilde{\rho})\geq t\left( \lambda
_{1}(\tilde{\rho}+h)-\lambda _{1}(\tilde{\rho})\right) \geq 0,\text{ }%
\forall h\in \mathbb{L}_{\mu }^{2},\text{ }\tilde{\rho}+h\in M_{\lambda }.
\end{equation*}%
Consequently, by (\ref{eq:dirtil}),
\begin{equation}
-\iint_{\mathbb{R}^{N}\times \mathbb{R}^{N}}\frac{\tilde{u}(x)\tilde{u}(y)}{%
|x-y|^{\mu }}h(x,y)dxdy=\lim_{t\rightarrow 0^{+}}\frac{\lambda _{1}(\tilde{%
\rho}+th)-\lambda _{1}(\tilde{\rho})}{t}\geq 0,\text{ }\forall h\in \mathbb{L%
}_{\mu }^{2},\text{ }\tilde{\rho}+h\in M_{\lambda }.  \label{eq:lambgeq0}
\end{equation}%
Now, by $\tilde{\rho}(x,y)=\rho (x,y)-\tilde{u}(x)\tilde{u}(y)$ and (\ref%
{eq:lambgeq0}) we obtain
\begin{eqnarray*}
D_{\tilde{\rho}}\Vert \tilde{\rho}-\rho \Vert _{\mathbb{L}_{\mu
}^{2}}^{2}(h)|_{\rho =\tilde{\rho}+\tilde{u}(x)\tilde{u}(y)} &=&2\iint_{%
\mathbb{R}^{N}\times \mathbb{R}^{N}}\frac{(\tilde{\rho}(x,y)-\rho
(x,y))h(x,y)}{|x-y|^{\mu }}dxdy \\
&=&-2\iint_{\mathbb{R}^{N}\times \mathbb{R}^{N}}\frac{\tilde{u}(x)\tilde{u}%
(y)}{|x-y|^{\mu }}h(x,y)dxdy\geq 0
\end{eqnarray*}%
for all $h\in \mathbb{L}_{\mu }^{2}$ such that $\tilde{\rho}+h\in M_{\lambda
}$.

By the Taylor expansion,
\begin{equation*}
\|\tilde{\rho} + h - \rho\|_{\mathbb{L}_{\mu}^2}^2 = \|\tilde{\rho} -
\rho\|_{\mathbb{L}_{\mu}^2}^2 + D_{\tilde{\rho}} \|\tilde{\rho} - \rho\|_{%
\mathbb{L}_{\mu}^2}^2 (h) + o(\|h\|_{\mathbb{L}_{\mu}^2})
\end{equation*}
as $\|h\|_{\mathbb{L}_{\mu}^2} \to 0$. Therefore, for sufficiently small $%
\|h\|_{\mathbb{L}_{\mu}^2}$ with $\tilde{\rho} + h \in M_{\lambda}$,
\begin{equation*}
\|\tilde{\rho} + h - \rho\|_{\mathbb{L}_{\mu}^2}^2 \geq \|\tilde{\rho} -
\rho\|_{\mathbb{L}_{\mu}^2}^2.
\end{equation*}
This shows that $\tilde{\rho}$ is a local minimizer of $\mathcal{P}(\theta)
= \|\theta - \rho\|_{\mathbb{L}_{\mu}^2}^2$ on $M_{\lambda}$.
\end{proof}

Now, we conclude the proof of Theorem~\ref{thm2}. According to Proposition~%
\ref{propGS}, equation (\ref{eq:M}) possesses the positive ground state $%
\hat{u}_{g}$ in $\mathbb{R}^{N}$. Suppose there exists another ground state $%
\tilde{u}\neq \hat{u}_{g}$ of equation (\ref{eq:M}) that is also positive
a.e. in $\mathbb{R}^{N}$.

By Proposition~\ref{Prop:4}, the functions
\begin{equation*}
\tilde{\rho}(x,y):=\bar{\rho}(x,y)-\tilde{u}(x)\tilde{u}(y)\quad \text{and}%
\quad \hat{\rho}_{g}(x,y):=\bar{\rho}(x,y)-\hat{u}_{g}(x)\hat{u}_{g}(y)
\end{equation*}%
are local minimizers of $\mathcal{P}$ on the manifold $\tilde{M}_{\lambda }$%
. Since both the functional $\mathcal{P}$ and the manifold $\tilde{M}%
_{\lambda }$ are strictly convex, the minimizer of $\mathcal{P}$ on $\tilde{M%
}_{\lambda }$ is unique. Thus, $\tilde{\rho}=\hat{\rho}_{g}$, which implies $%
\tilde{u}=\hat{u}_{g}$. Hence, the ground state is unique. Similarly,
\begin{equation*}
\hat{\rho}(x,y)=\bar{\rho}(x,y)-\hat{u}(x)\hat{u}(y)
\end{equation*}%
is a minimizer of $\mathcal{P}$ on the manifold $\tilde{M}_{\lambda }$; it
follows that $\hat{\rho}(x,y)=\hat{\rho}_{g}$. Hence the principal solution $%
\hat{u}\in W$ coincides with the ground state of equation (\ref{eq:M}), and
therefore it is positive.

\section{Proof of Theorem \protect\ref{thm3}}

First we prove

\begin{lemma}
\label{lem51}

(i) The map $\mathbb{L}^{2}_{\mu} \ni \bar{\rho} \mapsto \hat{\rho}(\lambda,%
\bar{\rho}) \in \mathbb{L}^{2}_{\mu}$ is continuous for any $\lambda \geq
\lambda_1(\bar{\rho})$;

(ii) The map $[\lambda _{1}(\bar{\rho}),+\infty )\ni \lambda \mapsto \hat{%
\rho}(\lambda ,\bar{\rho})\in \mathbb{L}_{\mu }^{2}$ is continuous for any $%
\bar{\rho}\in \mathbb{L}_{\mu }^{2}$.
\end{lemma}

\begin{proof}
Let us prove (i). Assume that $\lambda \in (\lambda _{1}(\bar{\rho}),+\infty
)$, and $\rho _{n}\rightarrow \bar{\rho}$ in $\mathbb{L}_{\mu }^{2}$ as $%
n\rightarrow \infty $. Recall that, by Proposition~\ref{prop2}, the
functional $\lambda _{1}(\cdot )$ is upper semicontinuous on $\mathbb{L}%
_{\mu }^{2}$. Hence $\lambda >\lambda _{1}(\rho _{n})$ for all sufficiently
large $n$, and therefore, by Lemma \ref{lem1} there exists a unique
minimizer $\hat{\rho}(\rho _{n})$ of
\begin{equation*}
\hat{\mathcal{P}}(\lambda ,\rho _{n}):=\Vert \rho _{n}-\hat{\rho}(\rho
_{n})\Vert _{\mathbb{L}_{\mu }^{2}}^{2}=\min \left\{ \Vert \rho _{n}-\rho
\Vert _{\mathbb{L}_{\mu }^{2}}^{2}\text{ }|\text{\ }\rho \in \mathbb{L}_{\mu
}^{2},\text{ }\lambda _{1}(\rho )\geq \lambda \right\} .
\end{equation*}

Using that $\lambda _{1}(\rho )$ is upper semicontinuous on $\mathbb{L}_{\mu
}^{2}$, it follows similarly to the proof of Proposition \ref{prop2} that $%
\hat{\mathcal{P}}(\lambda ,\cdot )$ is an upper semicontinuous functional on
$\mathbb{L}_{\mu }^{2}$ and therefore
\begin{equation}
\hat{\mathcal{P}}(\lambda ,\bar{\rho})\geq \limsup_{n\rightarrow +\infty }%
\hat{\mathcal{P}}(\lambda ,\rho _{n})\equiv \limsup_{n\rightarrow +\infty
}\Vert \hat{\rho}(\rho _{n})-\rho _{n}\Vert _{\mathbb{L}_{\mu }^{2}}.
\label{eq:mathcaP_conv}
\end{equation}%
Hence, the sequence $(\hat{\rho}(\rho _{n}))$ is bounded in $\mathbb{L}_{\mu
}^{2}$, and therefore, the reflexivity of $\mathbb{L}_{\mu }^{2}$ yields
that there exists a subsequence, still denoted by $(\hat{\rho}(\rho _{n}))$,
such that $\hat{\rho}(\rho _{n})\rightharpoonup \check{\rho}$ weakly in $%
\mathbb{L}_{\mu }^{2}$ for some $\check{\rho}\in \mathbb{L}_{\mu }^{2}$.
Analyzing similarly to the proof of Lemma \ref{lem1} we conclude that $%
\check{\rho}\neq 0$ and $\check{\rho}\in \tilde{M}_{\lambda }$.

It is easily seen that
\begin{equation}
\Vert \check{\rho}-\bar{\rho}\Vert _{\mathbb{L}_{\mu }^{2}}\leq
\liminf_{n\rightarrow +\infty }\Vert \hat{\rho}(\rho _{n})-\rho _{n}\Vert _{%
\mathbb{L}_{\mu }^{2}}.  \label{eq:52}
\end{equation}%
Hence by (\ref{eq:mathcaP_conv}) we derive
\begin{equation*}
\Vert \check{\rho}-\bar{\rho}\Vert _{\mathbb{L}_{\mu }^{2}}\leq \hat{%
\mathcal{P}}(\lambda ,\bar{\rho}).
\end{equation*}%
Since $\check{\rho}\in \tilde{M}_{\lambda }$, equality must hold: $\Vert
\check{\rho}-\bar{\rho}\Vert _{\mathbb{L}_{\mu }^{2}}=\hat{\mathcal{P}}%
(\lambda ,\bar{\rho})$. By the uniqueness of the minimizer of (\ref{eq:p}),
this implies that $\check{\rho}=\hat{\rho}(\bar{\rho})$. Moreover, (\ref%
{eq:mathcaP_conv}) and (\ref{eq:52}) yield
\begin{equation*}
\Vert \hat{\rho}(\bar{\rho})-\bar{\rho}\Vert _{\mathbb{L}_{\mu
}^{2}}=\lim_{n\rightarrow +\infty }\Vert \hat{\rho}(\rho _{n})-\rho
_{n}\Vert _{\mathbb{L}_{\mu }^{2}},
\end{equation*}%
and thus,
\begin{equation*}
\hat{\rho}(\rho _{n})\rightarrow \hat{\rho}(\bar{\rho})\quad \text{in}\quad
\mathbb{L}_{\mu }^{2}.  \label{eq:53}
\end{equation*}%
Since $\hat{\rho}(\bar{\rho})$ is uniquely defined, this implies that the
map $\mathbb{L}_{\mu }^{2}\ni \bar{\rho}\mapsto \hat{\rho}(\lambda ,\bar{\rho%
})\in \mathbb{L}_{\mu }^{2}$ is continuous.

Now we prove (ii). Assume that $\bar{\rho}\in \mathbb{L}_{\mu }^{2}$ and $%
\lambda >\lambda _{1}(\bar{\rho})$, and $\lambda _{n}\rightarrow \lambda $
as $n\rightarrow \infty $. Then $\lambda _{n}>\lambda _{1}(\bar{\rho})$ for
all sufficiently large $n$. Hence, by Lemma \ref{lem1} there exists a unique
minimizer $\hat{\rho}(\lambda _{n}):=\hat{\rho}(\lambda _{n},\bar{\rho})$ of
\begin{equation*}
\hat{\mathcal{P}}(\lambda _{n},\bar{\rho}):=\Vert \bar{\rho}-\hat{\rho}%
(\lambda _{n})\Vert _{\mathbb{L}_{\mu }^{2}}^{2}=\min \left\{ \Vert \bar{\rho%
}-\rho \Vert _{\mathbb{L}_{\mu }^{2}}^{2}\text{ }|\text{\ }\rho \in \mathbb{L%
}_{\mu }^{2},\text{ }\lambda _{1}(\rho )\geq \lambda _{n}\right\} .
\end{equation*}%
As before, it can be shown that there exists a weak limit point $\breve{\rho}%
\in \mathbb{L}_{\mu }^{2}\setminus \{0,\bar{\rho}\}$ of the sequence $(\hat{%
\rho}(\lambda _{n}))$ in $\mathbb{L}_{\mu }^{2}$, and
\begin{equation*}
\hat{\mathcal{P}}(\lambda ,\bar{\rho})\geq \limsup_{n\rightarrow +\infty }%
\hat{\mathcal{P}}(\lambda _{n},\bar{\rho}).
\end{equation*}%
Therefore,
\begin{equation*}
\Vert \breve{\rho}-\bar{\rho}\Vert _{\mathbb{L}_{\mu }^{2}}\leq
\liminf_{n\rightarrow +\infty }\Vert \hat{\rho}(\lambda _{n})-\bar{\rho}%
\Vert _{\mathbb{L}_{\mu }^{2}}\leq \hat{\mathcal{P}}(\lambda ,\bar{\rho}).
\label{eq:58}
\end{equation*}

Since the functional $\lambda(\cdot)$ is upper semicontinuous on $\mathbb{L}%
_{\mu}^{2}$, it follows that
\begin{equation*}
\lambda(\breve{\rho})\geq \limsup_{n\to+\infty}\lambda(\hat{\rho}%
(\rho_n))\geq \lambda.
\end{equation*}
Thus, $\breve{\rho} \in \tilde{M}_{\lambda}$, and therefore
\begin{equation*}
\|\breve{\rho} -\bar{\rho}\|_{\mathbb{L}^2_{\mu }}= \liminf_{n\to +\infty}\|%
\hat{\rho}(\rho_n)-\bar{\rho}\|_{\mathbb{L}^2_{\mu }}= \hat{\mathcal{P}}%
(\lambda,\bar{\rho}).
\end{equation*}
This indicates that $\hat{\rho}(\lambda_n) \to \breve{\rho}$ in $\mathbb{L}%
^2_{\mu}$ as $n \to +\infty$. The remainder of the proof follows as before.
\end{proof}

Let us prove assertion (i) of Theorem \ref{thm3}. Let $\lambda \in (\lambda
_{1}(\bar{\rho}),+\infty )$ and $\rho _{n}\rightarrow \bar{\rho}$ in $%
\mathbb{L}_{\mu }^{2}$ as $n\rightarrow \infty $. By Lemma \ref{lem51} there
exists a subsequence, again denoted by $(\hat{\rho}(\rho _{n}))$, and
normalized eigenfunctions $\phi _{1}[\hat{\rho}(\rho _{n})]$ of $\mathcal{L}[%
\hat{\rho}(\rho _{n})]$, and $\phi _{1}[\hat{\rho}(\bar{\rho})]$ of $%
\mathcal{L}[\hat{\rho}(\bar{\rho})]$ such that
\begin{eqnarray*}
&&\phi _{1}[\hat{\rho}(\rho _{n})]\rightarrow \phi _{1}[\hat{\rho}(\bar{\rho}%
)]\text{ in }W\text{ as }n\rightarrow +\infty ,  \label{eq:phi_conv} \\
&&\lambda _{1}(\hat{\rho}(\rho _{n}))\rightarrow \lambda _{1}(\hat{\rho}(%
\bar{\rho}))\text{ as }n\rightarrow +\infty .  \label{eq:lamb_conv}
\end{eqnarray*}%
Note that by Lemma \ref{lem2}
\begin{eqnarray*}
&&\phi _{1}[\hat{\rho}(\rho _{n})](x)=\frac{\hat{u}(\lambda ,\rho _{n})(x)}{%
\Vert \hat{u}(\lambda ,\rho _{n})\Vert _{L^{2}}}\quad \text{a.e. in }\mathbb{%
R}^{N},\text{ }n=1,\ldots , \\
&&\lambda =\lambda _{1}(\hat{\rho}(\bar{\rho})).
\end{eqnarray*}%
Denote $t_{n}=\Vert \hat{u}(\lambda ,\rho _{n})\Vert _{W}$, $n=1,\ldots $.
Then
\begin{eqnarray*}
&&\int_{\mathbb{R}^{N}}(|\nabla \phi _{1}[\hat{\rho}(\rho
_{n})]|^{2}+V(x)|\phi _{1}[\hat{\rho}(\rho _{n})]|^{2})dx-\int_{\mathbb{R}%
^{N}}\int_{\mathbb{R}^{N}}\frac{\hat{\rho}(\rho _{n})(x,y)\phi _{1}[\hat{\rho%
}(\rho _{n})](x)\phi _{1}[\hat{\rho}(\rho _{n})](y)}{|x-y|^{\mu }}dxdy \\
&&-\lambda _{1}(\hat{\rho}(\bar{\rho}))\int_{\mathbb{R}^{N}}|\phi _{1}[\hat{%
\rho}(\rho _{n})]|^{2}dx-t_{n}^{2}\int_{\mathbb{R}^{N}}\int_{\mathbb{R}^{3}}%
\frac{\phi _{1}[\hat{\rho}(\rho _{n})]^{2}(x)\phi _{1}[\hat{\rho}(\rho
_{n})]^{2}(y)}{|x-y|^{\mu }}dxdy=0,\quad n=1\ldots .
\end{eqnarray*}%
Suppose $t_{n}\rightarrow 0$. Then passing to the limit we obtain
\begin{equation*}
-\Delta \phi _{1}[\hat{\rho}(\bar{\rho})]+V(x)\phi _{1}[\hat{\rho}(\bar{\rho}%
)]-\int_{\mathbb{R}^{N}}\frac{\rho _{0}(x,y)\phi _{1}[\hat{\rho}(\bar{\rho}%
)](y)}{|x-y|^{\mu }}dy=\lambda _{1}(\hat{\rho}(\bar{\rho}))\phi _{1}[\hat{%
\rho}(\bar{\rho})].
\end{equation*}%
On the other hand, $\phi _{1}[\hat{\rho}(\bar{\rho})]$ is an eigenfunction
corresponding to the principal eigenvalue of $\mathcal{L}[\hat{\rho}(\bar{%
\rho})]$, and therefore
\begin{equation*}
-\Delta \phi _{1}[\hat{\rho}(\bar{\rho})]+V(x)\phi _{1}[\hat{\rho}(\bar{\rho}%
)]-\int_{\mathbb{R}^{N}}\frac{\hat{\rho}(\bar{\rho})(x,y)\phi _{1}[\hat{\rho}%
(\bar{\rho})](y)}{|x-y|^{\mu }}dy=\lambda _{1}(\hat{\rho}(\bar{\rho}))\phi
_{1}[\hat{\rho}(\bar{\rho})],  \label{eq:eigenP2}
\end{equation*}%
where
\begin{equation*}
\hat{\rho}(\bar{\rho})(x,y)=\bar{\rho}(x,y)+\eta \phi _{1}[\hat{\rho}(\bar{%
\rho})](x)\phi _{1}[\hat{\rho}(\bar{\rho})](y),
\end{equation*}%
with some $\eta >0$. By subtraction, we derive
\begin{equation*}
\eta \int_{\mathbb{R}^{N}}\frac{(\phi _{1}[\hat{\rho}(\bar{\rho})](y))^{2}}{%
|x-y|^{\mu }}dy\,\phi _{1}[\hat{\rho}(\bar{\rho})](x)=0,\quad x\in \mathbb{R}%
^{N},
\end{equation*}%
which implies $\phi _{1}[\hat{\rho}(\bar{\rho})]=0$ in $\mathbb{R}^{N}$.
Thus, we get a contradiction. By similar reasoning by contradiction, we
conclude that the sequence $\{t_{n}\}$ is bounded. Hence, summarizing the
above, we infer that
\begin{equation*}
\hat{u}(\lambda ,\rho _{n})\rightarrow \hat{u}(\lambda ,\bar{\rho})\quad
\text{in}\quad W\quad \text{as}\quad n\rightarrow +\infty .
\end{equation*}%
Theorem \ref{thm1} implies that $\hat{u}(\lambda ,\bar{\rho})$ is a
principal solution when $\lambda >\lambda _{1}(\bar{\rho})$. By
contradiction, we deduce that $\hat{u}(\lambda ,\bar{\rho})=0$ when $\lambda
=\lambda _{1}(\bar{\rho})$.

The proof of (ii) is similar to that of (i). The proof of (iii) is a direct
consequence of the assertions (i), (ii), and Theorem \ref{thm2}.

\section{Proof of Theorem \protect\ref{thm4}}

Let $\bar{\rho}\in \mathbb{L}_{\mu }^{2}$. Consider the function $[\lambda
_{1}(\bar{\rho}),+\infty )\ni \lambda \mapsto \hat{\mathcal{P}}(\lambda ,%
\bar{\rho})=\Vert \bar{\rho}-\hat{\rho}(\lambda ,\bar{\rho})\Vert _{\mathbb{L%
}_{\mu }^{2}}^{2}$ as defined by (\ref{eq:p}). Lemma \ref{lem51} implies
that this function is continuous on $[\lambda _{1}(\bar{\rho}),+\infty )$
and $\hat{\mathcal{P}}(\lambda _{1}(\bar{\rho}),\bar{\rho})=0$. Furthermore,
$\hat{\mathcal{P}}(\cdot ,\bar{\rho})$ is strong monotone increasing
function. Indeed, $\tilde{M}_{\lambda ^{\prime }}\subset \tilde{M}_{\lambda
} $ if $\lambda ^{\prime }>\lambda $, and therefore, (\ref{MinP}) yields $%
\hat{\mathcal{P}}(\lambda ^{\prime },\bar{\rho})\geq \hat{\mathcal{P}}%
(\lambda ,\bar{\rho})$. To obtain a contradiction, suppose that $\hat{%
\mathcal{P}}(\lambda ^{\prime },\bar{\rho})=\hat{\mathcal{P}}(\lambda ,\bar{%
\rho})$. Note that $\hat{\rho}(\lambda ^{\prime },\bar{\rho})\in \tilde{M}%
_{\lambda ^{\prime }}\subset \tilde{M}_{\lambda }$, and therefore
\begin{equation*}
\Vert \bar{\rho}-\hat{\rho}(\lambda ^{\prime },\bar{\rho})\Vert _{\mathbb{L}%
_{\mu }^{2}}^{2}=\hat{\mathcal{P}}(\lambda ^{\prime },\bar{\rho})=\hat{%
\mathcal{P}}(\lambda ,\bar{\rho})=\min \left\{ \Vert \rho -\bar{\rho}\Vert _{%
\mathbb{L}_{\mu }^{2}}^{2}\mid \rho \in \tilde{M}_{\lambda }\right\} .
\end{equation*}%
Hence, the uniqueness of the minimizer of (\ref{MinP}) implies a
contradiction.

Thus, $\hat{\mathcal{P}}(\lambda, \bar{\rho}) \to C\leq +\infty$ as $\lambda
\to +\infty$. Suppose, contrary to our claim, that $C<+\infty$. Then there
exists a sequence $\hat{\rho}^n:=\hat{\rho}(\lambda_n, \bar{\rho})$, $%
n=1,\ldots$ such that $\hat{\mathcal{P}}(\lambda_n, \bar{\rho})=\| \bar{\rho}
- \hat{\rho}^n \|_{\mathbb{L}_{\mu}^{2}}^{2} \to C$ as $n\to +\infty$. As in
the proof of Theorem \ref{thm1} this implies that there exists a weak in $%
\mathbb{L}_{\mu}^{2}$ limit point $\tilde{\rho} \in \mathbb{L}_{\mu}^{2}$ of
$(\hat{\rho}^n)$. This by weakly upper semicontinuity of $\lambda_1(\rho)$
implies a contradiction: $+\infty =\limsup_{n\to +\infty} \lambda_1(\hat{\rho%
}^n)\leq \lambda_1(\tilde{\rho})<+\infty$.

Thus, $\hat{\mathcal{P}}(\cdot ,\bar{\rho})$ is strong monotone increasing
function, and therefore, there exists a continuous inverse function $\hat{%
\mathcal{P}}^{-1}(\cdot ,\bar{\rho})$. Define $\lambda ^{\kappa }(\bar{\rho}%
)=\hat{\mathcal{P}}^{-1}(\kappa ,\bar{\rho})$, $\forall \kappa >0$. Then, $%
\kappa =\hat{\mathcal{P}}(\lambda ^{\kappa }(\bar{\rho}),\bar{\rho})$. Let
us show that $\check{\rho}=\hat{\rho}(\lambda ^{\kappa }(\bar{\rho}),\bar{%
\rho})$ is a solution of (\ref{eq-10}). Conversely, suppose that there
exists $\tilde{\rho}\in L_{\mu }^{2}\setminus \{\hat{\rho}(\lambda ^{\kappa
}(\bar{\rho}),\bar{\rho})\}$ such that
\begin{equation*}
\Vert \bar{\rho}-\tilde{\rho}\Vert _{\mathbb{L}_{\mu }^{2}}^{2}=\kappa \text{
and }\lambda _{1}(\tilde{\rho})>\lambda ^{\kappa }(\bar{\rho}).
\end{equation*}%
Since $\kappa =\hat{\mathcal{P}}(\lambda ^{\kappa }(\bar{\rho}),\bar{\rho})$%
, $\tilde{\rho}$ is a minimizer of (\ref{MinP}) with $\lambda =\lambda
^{\kappa }(\bar{\rho})$. This is a contradiction, since (\ref{MinP}) has a
unique minimizer $\hat{\rho}(\lambda ^{\kappa }(\bar{\rho}),\bar{\rho})$ and
by the assumption $\tilde{\rho}\neq \hat{\rho}(\lambda ^{\kappa }(\bar{\rho}%
),\bar{\rho})$. Thus, $\lambda _{1}^{\kappa }(\bar{\rho})=\lambda ^{\kappa }(%
\bar{\rho})=\hat{\mathcal{P}}^{-1}(\kappa ,\bar{\rho})$, $\check{\rho}=\hat{%
\rho}(\lambda _{1}(\bar{\rho}),\bar{\rho})$ and $\lambda _{1}^{\kappa }(\bar{%
\rho})=\lambda _{1}(\check{\rho})>\lambda _{1}(\bar{\rho})$.

The remainder of the proof follows directly from Theorems \ref{thm1} and \ref%
{thm2}.

\section{Concluding remarks}

\label{sec:conclusion}

This study serves a dual purpose: it aims not only to solve the nonlinear
equation \eqref{eq:M} but also to develop a novel approach in the theory of
inverse problems. More generally, we address the inverse problem of
reconstructing the density function $\rho$ of a linear operator $\mathcal{L}%
[\rho]$ from a finite set of observed eigenvalues $\{\lambda_i(\rho)%
\}_{i=1}^m$, where $1 \leq m < \infty$. Such problems are inherently
ill-posed, as a finite amount of spectral data typically admits infinitely
many solutions (see, e.g., \cite{chadan, glad}).

A novel approach to these challenges has been proposed in \cite{ValIl_JDE}.
The method is predicated on the assumption that \emph{a priori} information
about the system is available. Specifically, for the operator $\mathcal{L}%
[\rho]$, we assume the existence of a reference function $\bar{\rho }\in
\mathbb{L}^2_{\mu}$ that approximates the unknown density $\rho$.

This framework leads to the \emph{inverse optimal problem} (IOP) \cite%
{ValIl_JDE, ValIl_PhysD}:

($\mathcal{P}(m)$): \textit{Given a reference density $\rho _{0}\in \mathbb{L%
}_{\mu }^{2}$ and a finite set of target eigenvalues $\{\lambda
_{i}\}_{i=1}^{m}$ with $1\leq m<\infty $, find a density $\hat{\rho}$ that
minimizes the distance to $\rho _{0}$ in a prescribed norm, subject to the
constraints $\lambda _{i}(\hat{\rho})=\lambda _{i}$ for $i=1,\dots ,m$.}

The inference procedure used for the single equation (\ref{eq:M}) can be
generalized to the multiparameter case ($\mathcal{P}(m)$). In this setting,
one obtains the following coupled system of nonlinear equations associated
with problem ($\mathcal{P}(m)$):

\begin{equation*}
\begin{cases}
-\Delta u_{1}+V(\mathbf{x})u_{1}-S[\rho ]u_{1}+\sum_{k=1}^{m}\sigma
_{k}\left( |\mathbf{x}|^{-\mu }\ast |u_{k}|^{2}\right) u_{1}=\lambda
_{1}u_{1}, \\[6pt]
\hspace{6cm}\vdots \\[6pt]
-\Delta u_{m}+V(\mathbf{x})u_{m}-S[\rho ]u_{m}+\sum_{k=1}^{m}\sigma
_{k}\left( |\mathbf{x}|^{-\mu }\ast |u_{k}|^{2}\right) u_{m}=\lambda
_{m}u_{m},%
\end{cases}%
\text{ }\mathbf{x}\in \mathbb{R}^{N},
\end{equation*}%
where $\rho $ denotes the associated two-point density function, and the
coefficients $\sigma _{k}\in \mathbb{R}$ ($k=1,\dots ,m$) characterize the
type and strength of the nonlocal interactions.

It should be emphasized that the uniqueness result for the ground state
established in Theorem \ref{thm2} remains valid even in the case $\rho =0$,
i.e., for the equation
\begin{equation*}
-\Delta u+V(x)u+\left( |x|^{-\mu }\ast |u|^{2}\right) u=\lambda u,\quad
\forall x\in \mathbb{R}^{N}.
\end{equation*}%
To the best of our knowledge, this result has not been established
previously in the literature.

The inverse optimal problem framework developed in this work offers several
advantages for both theoretical analysis and numerical computation. First,
by reformulating the nonlinear PDE as a constrained optimization problem, we
obtain a variational characterization of solutions that is particularly
amenable to numerical approximation techniques, including modern machine
learning approaches such as the Deep Ritz method. Second, the explicit
reconstruction formulas for $\hat{\rho}$ and $\hat{u}$ in terms of spectral
data provide a direct computational pathway for solving inverse problems in
quantum systems and other contexts where only partial spectral information
is available. Finally, the extension to multiparameter problems opens the
door to applications in multi-component quantum systems and coupled
nonlinear wave phenomena. Future work will explore these computational
implementations and applications to specific physical systems.

\section*{Acknowledgments}

Juntao Sun was supported by the National Natural Science Foundation of China
(Grant No. 12371174). Shuai Yao was supported by the National Natural
Science Foundation of China (Grant No. 12501218) and Shandong Provincial
Natural Science Foundation (Grant No. ZR2025QC1512).

\end{document}